\documentclass[a4paper,12pt]{article}
\usepackage{latexsym}
\usepackage{amsmath}
\usepackage{amssymb}
\usepackage{enumerate}
\usepackage{bm}
\usepackage{mathrsfs}

\usepackage[anchorcolor=blue,%
 bookmarks=true,%
 bookmarksnumbered=true,%
 colorlinks=true,%
 citecolor=cyan,%
 linkcolor=blue,%
]{hyperref}
\usepackage{theorem}
\newtheorem{theorem}{Theorem}[section]
\newtheorem{proposition}{Proposition}[section]
\newtheorem{lemma}{Lemma}[section]

\theorembodyfont{\rmfamily}
\newtheorem{remark}{Remark}[section]

\newtheorem{definition}{Definition}[section]
\newtheorem{proof}{\textmd{\textit{Proof.}}}

\makeatletter

\@addtoreset{equation}{section}
\makeatother

\newcommand{\qedd}{\hfill \square}
\newcommand{\ve}{\varepsilon}
\newcommand{\vez}{\varepsilon}
\newcommand{\del}{\partial}
\newcommand{\lra}{\longrightarrow}

\def\lf{\left}
\def\r{\right}

\newcommand{\R}{\ensuremath{\mathbb{R}}}

\newcommand{\fB}{\ensuremath{\mathfrak{B}}}

\newcommand{\bs}{\ensuremath{\mathbf{s}}}
\newcommand{\bt}{\ensuremath{\mathbf{t}}}
\newcommand{\bK}{\ensuremath{\mathbf{K}}}

\newcommand{\sA}{\ensuremath{\mathsf{A}}}
\newcommand{\sB}{\ensuremath{\mathsf{B}}}
\newcommand{\sC}{\ensuremath{\mathsf{C}}}

\newcommand{\sI}{\ensuremath{\mathsf{I}}}
\newcommand{\sJ}{\ensuremath{\mathsf{J}}}
\newcommand{\sO}{\ensuremath{\mathsf{O}}}
\newcommand{\sR}{\ensuremath{\mathsf{R}}}

\def\td{\mathrm{d}}

\def\Ric{\mathop{\mathrm{Ric}}\nolimits}
\def\trace{\mathop{\mathrm{trace}}\nolimits}
\def\tr{\mathop{\mathrm{tr}}\nolimits}

\def\Fut{\mathop{\mathrm{Fut}}\nolimits}

\newcommand{\noz}{\nonumber}

\newcommand{\wz}[1]{\widetilde{#1}}

\def\hs{\hspace{.3cm}}
\title{Volume comparison theorems in Finsler spacetimes}

\author{Yufeng LU}

\date{\today}
\begin{document}

\maketitle

\begin{abstract}
In a $(1+n)$-dimensional Lorentz--Finsler manifold
with $N$-Bakry--\'Emery Ricci curvature bounded from below
where $N\in(n,\infty]$, using the Riccati equation techniques,
we establish the Bishop--Gromov volume comparison theorem for the
so-called standard sets for comparisons in Lorentzian volumes (SCLVs).
We also establish the G\"unther volume comparison theorem for SCLVs when
the flag curvature is bounded above.
\end{abstract}

\tableofcontents
\section{Introduction}
Comparison theorems are classical subjects in Lorentzian geometry
in connection with, for example, physical convergence conditions
(nonnegative Ricci curvature), Raychaudhuri equations and
singularity theorems. Recently Lorentzian comparison theory
is attracting growing interests also from the synthetic
geometric viewpoint, we refer to \cite{AB,KuSa} for triangle
comparison theorems, and to \cite{Mc,MS} for connections with
optimal transport theory and the curvature-dimension condition.

Motivated by them, we initiated comparison theory in the weighted
Lorentz--Finsler setting in \cite{LMO}, where we studied
Raychaudhuri equations and singularity theorems.
This article is a continuation and extend some volume comparison
theorems to the Lorentz--Finsler case. The Bishop--Gromov comparison
theorem plays an important role in global analysis, especially
in comparison geometry of Riemannian manifolds. The original
version of the theorem assumes a lower bound for the Ricci curvature
of a Riemannian manifold. (See Chavel's textbook \cite{Ch},
for example.)
It has been developed by various papers in different aspects.
Generalizations of the Bishop--Gromov comparison theorem for weighted
Riemannian manifolds under the $N$-Bakry--\'Emery Ricci curvature condition
were completed by Qian \cite{Qian} for $N\in[n,\infty)$ and
considered by Wei and Wylie \cite{WW} for $N=\infty$
with some additional assumptions on the weight function.
Shen \cite{Sh} extended this comparison theorem
to the Finsler setting with unweighted Ricci curvature
and $S$-curvature conditions.
After Ohta \cite{Oint} introduced the weighted Ricci curvature
for Finsler manifolds, \cite{Oint} and \cite{Yin} gave proofs
to the Bishop--Gromov comparison theorems under weighted
Ricci curvature conditions for $N\in[n,\infty)$ and $N=\infty$
with some assumptions on the $S$-curvature
of the manifold, respectively. Moreover,
by using a differential inequality for an elliptic second
order differential operator acting on distance functions,
\cite{BQ} deduced Bishop--Gromov comparison theorems and diameter bounds
without use of the theory of Jacobi fields.
In the Lorentzian settings, Ehrlich, Jung and Kim \cite{EJK} studied the
Bishop--Gromov theorems for compact geodesic wedges
in globally hyperbolic spacetimes. Later on,
Ehrlich and S\'anchez \cite{ES} defined a more natural set on which
the assumption of global hyperbolicity was no more necessary in
either the Bishop--Gromov's or the G\"unther's volume comparisons.
For other generalizations, we refer to \cite{Lee, Zhu, Kim}
and \cite{Wei} for a survey on this topic.

Throughout this paper the function $\bs_\kappa$ is the solution to the
differential equation $f''+\kappa f=0$ with $f(0)=0$ and $f'(0)=1$, i.e.
\begin{align*}
\bs_\kappa(t):=\begin{cases}
\dfrac1{\sqrt{\kappa}}\sin\left(\sqrt{\kappa}t\right),&\textrm{for }\kappa>0;\\
t,&\textrm{for }\kappa=0;\\
\dfrac1{\sqrt{-\kappa}}\sinh\left(\sqrt{-\kappa}t\right),&\textrm{for }\kappa<0.
\end{cases}
\end{align*}
Our main theorems are as follows.
\begin{theorem}\label{thm:vc}
Let $(M,L,\rho)$ be a Finsler spacetime equipped with
a measure $\rho$ on $M$. Let $U_x$ be a
SCLV at $x\in M$. Assume that $\Ric_N(v)\ge c$ for some
$c\in\R$ and for all unit timelike radial vectors $v$,
where $N\in(n,\infty).$ If $\bt_{U_x}$ is
constant on $U_1$, then we have
\begin{align*}
\frac{\rho\left(U_x(r)\right)}{\rho\left(U_x(R)\right)}
\ge\frac{\int_0^{rT_x}\bs_{c/N}^N(t)\,dt}{\int_0^{RT_x}\bs_{c/N}^N(t)\,dt},
\end{align*}
for all $0<r\le R\le 1$, where $T_x=\bt_x$ if $c\le0$;
$T_x=\min\{\bt_x,\pi\sqrt{\frac Nc}\}$ if $c>0$.
\end{theorem}

We also refer to the flag curvature $\bK$ in Definition \ref{df:curv}
and the weight function $\psi$ in Definition \ref{def:RicN}.
The G\"unther volume theorem could be stated as follows
with some notations introduced in the end of Section \ref{sc:BG}.

\begin{theorem}\label{thm:vc-flag}
Let $(M,L,\rho)$ be a Finsler spacetime equipped with a measure $\rho$ on $M$. Let $U_x$ be a SCLV
at $x\in M$. Assume that the flag curvature $\bK(\pi)\le -c$
for some $c\ge0$ and for all radially timelike planes $\pi$.
If $\psi\le k$, then
\begin{align*}
\rho(U_x)\ge e^{-k}\sigma\lf(\wz U_1\r)\int_0^{\bt_x}\bs_{-c}^n(t)\td t,
\end{align*}
where $\sigma$ is the area measure on $\Fut_1(x)$ induced
from $\rho$.
\end{theorem}

In the next section, we give some preliminaries of Lorentz--Finsler
manifolds. Theorems \ref{thm:vc} and \ref{thm:vc-flag}
are shown in Sections \ref{sc:BG} and \ref{sc:Gn}, respectively.
We also consider the case $N=\infty$ in Section \ref{sc:Ninf}
along the lines of \cite{WW}.


\section{Lorentz--Finsler manifolds}\label{sc:LF-mfd}
Let $M$ be a connected paracompact $C^\infty$-manifold of dimension $1+n$. We also refer to \cite{BCS,Shlec} for
Finsler geometry in the positive-definite case.
\subsection{Metric}
In this article, Beem's definition of Lorentz--Finsler manifolds \cite{Be}
is employed.
\begin{definition}[Lorentz--Finsler structure]\label{df:FLstr}
A \emph{Lorentz--Finsler structure} of $M$ will be a function
$L\colon TM \lra \R$ satisfying the following conditions:
\begin{enumerate}[(1)]
\item $L \in C^{\infty}(TM \setminus \{0\})$;
\item $L(cv)=c^2 L(v)$ for all $v \in TM$ and $c>0$;
\item For any $v \in TM \setminus \{0\}$, the symmetric matrix
\begin{equation}\label{eq:g_ij}
g_{\alpha \beta}(v) := \frac12\frac{\del^2 L}{\del v^\alpha \del v^\beta}(v),
 \quad \alpha,\beta=0,1,\ldots,n,
\end{equation}
is non-degenerate with signature $(-,+,\ldots,+)$.
\end{enumerate}
A pair $(M,L)$ is then said to be a \emph{Lorentz--Finsler manifold}
or a \emph{Lorentz--Finsler space}.
\end{definition}
Let $x\in M$ and $v\in T_xM\setminus\{0\}$. The matrix \eqref{eq:g_ij} induces a Lorentzian metric $g_v$ called the \emph{fundamental tensor} by
\begin{align*}
g_v\left(\left.\sum_{\alpha=0}^na^\alpha\frac\partial{\partial x^\alpha}
\right|_x,\left.\sum_{\beta=0}^nb^\beta\frac\partial{\partial x^\beta}
\right|_x\right)=\sum_{\alpha,\beta=0}^na^\alpha b^\beta g_{\alpha\beta}(v).
\end{align*}
Similar to the positive-definite case, the tensor $g_v$ for
$v\in TM\setminus\{0\}$ satisfies that
\begin{align}
g_v(v,v)=\sum_{\alpha,\beta=0}^nv^\alpha v^\beta g_{\alpha\beta}(v)=L(v).
\end{align}
\subsection{Causality}\label{ssc:causality}
\begin{definition}\label{df:time}
Let $(M,L)$ be a Lorentz--Finsler manifold.
A vector $v \in TM $ is said to be a \emph{timelike vector} if $L(v)<0$
and a \emph{null vector} if $L(v)=0$. A vector $v$ is said
to be \emph{lightlike} if it is null and non-zero. The
\emph{spacelike vectors} are those for which $L(v)>0$ or $v=0$.
The \emph{causal} (or \emph{non-spacelike}) vectors are those
which are lightlike or timelike ($L(v) \le 0$ and $v \neq 0$).
The set of timelike vectors will be denoted by
\begin{align*}
\Omega'_x:=\{ v \in T_xM \,|\, L(v)<0 \},
 \qquad \Omega' :=\bigcup_{x \in M} \Omega'_x.
\end{align*}
\end{definition}

\begin{definition}
Let $(M,L)$ be a Lorentz--Finsler manifold.
A continuous vector field $X$ on $M$ is said to be
\emph{timelike} if $L(X(x))<0$ for all $x \in M$.
If $(M,L)$ admits a timelike smooth vector field $X$,
then $(M,L)$ is said to be \emph{time oriented} by $X$,
or simply \emph{time oriented}. A time oriented
Lorentz--Finsler manifold is then said to be a \emph{Finsler spacetime}.
\end{definition}

Let $X_M$ be a fixed timelike smooth vector field.
A causal vector $v \in T_xM$ is said to be
\emph{future-directed} (with respect to $X_M$)
if it lies in the same connected component of
$\overline{\Omega'_x} \setminus \{0\}$ as $X_M(x)$.
Denote by $\Omega_x \subset \Omega'_x$ the set of all
future-directed timelike vectors that is a connected component of $\Omega_x'$ and make the
following notations,
\begin{align*}
\Omega :=\bigcup_{x \in M} \Omega_x, \qquad
 \overline{\Omega} :=\bigcup_{x \in M} \overline{\Omega}_x, \qquad
 \overline{\Omega} \setminus \{0\} :=\bigcup_{x \in M} (\overline{\Omega}_x \setminus \{0\}).
\end{align*}
A $C^1$-curve in $(M,L)$ is said to be \emph{timelike} (resp.\
\emph{causal, lightlike, spacelike})
if its tangent vectors are always timelike
(resp.\ causal, lightlike, spacelike). A causal curve is said to be \emph{future-directed} if its tangent vectors are always future-directed.

\begin{remark}
It is well-known that, in general, the number of connected components
of $\Omega_x'$ may be larger than 2, but in a reversible
Lorentz--Finsler manifold of dimension larger than 2, $\Omega_x'$
has exactly two connected components. See \cite{Be, Min3}.
In this special case, one may define past-directed vectors as well.
\end{remark}

Given distinct points $x,y \in M$, we write $x \ll y$ if there
is a future-directed timelike curve from $x$ to $y$. Similarly,
$x<y$ means that there is a future-directed causal curve from
$x$ to $y$, and $x \le y$ means that $x=y$ or $x<y$.
The \emph{chronological past} and \emph{future} of $x$ are
defined by
\begin{align*}
I^-(x):=\{y\in M\,|\,y\ll x\},\qquad I^+(x):=\{y\in M\,|\,x\ll y\},
\end{align*}
and the \emph{causal past} and \emph{future} by
\begin{align*}
J^-(x):=\{y\in M\,|\,y\le x\},\qquad J^+(x):=\{y\in M\,|\,x\le y\}.
\end{align*}
For a general set $S\subset M$, we define $I^+(S),I^-(S),J^+(S)$
and $J^-(S)$ analogously.
\begin{definition}[Causal convexity] Let $(M,L)$ be a Finsler spacetime.
An open set $U\subset M$ is said to be \emph{causally convex}
if no causal curve intersects $U$ in a disconncected set of its domain.
\end{definition}
Using these terminologies, several
causality conditions may be defined as follows.
\begin{definition}[Causality conditions]\label{df:causal}
Let $(M,L)$ be a Finsler spacetime.
\begin{enumerate}[(1)]
\item $(M,L)$ is said to be \emph{chronological}
if $x \notin I^+(x)$ for all $x\in M$.
\item We say that $(M,L)$ is \emph{causal} if there
is no closed causal curve.
\item $(M,L)$ is said to be \emph{strongly causal at
a point} $x\in M$, if $x$ has arbitrarily small causally
convex neighborhoods. $(M,L)$ is said to be \emph{strongly causal}
if it is strongly causal at every point $x\in M$.
\item We say that $(M,L)$ is \emph{globally hyperbolic}
if it is strongly causal and, for any $x,y \in M$, $J^+(x)
\cap J^-(y)$ is compact.
\end{enumerate}
\end{definition}
\begin{remark}
From the definitions, it is clear that strong causality implies
causality and a causal spacetime is chronological.
The chronological condition ensures that the spacetime is not compact.
Indeed, if we assume the contrary, since $\{I^+(x)\}_{x\in M}$ forms
an open cover of $M$, the existence of a finite subcover of $M$ and the fact that for any $x\in M$ we have $x\notin I^+(x)$ indicate that there must be a point in the future set of itself.
\end{remark}

\begin{remark}\label{rm:vc}
In volume comparison theorems in the positive-definite case, we usually
assume the completeness and compare the volumes of concentric balls.
In the Lorentzian case, however, since $\Fut_1(x)$ is noncompact and
``balls'' can have infinite volume, we need to restrict ourselves to a compact set
(like a SCLV). In \cite{EJK}, the authors assumed the global
hyperbolicity and the positivity of the injectivity radius within
relevant directions, however, the global hyperbolicity is not
fulfilled even by some standard examples like anti-de Sitter
spacetimes. Then the SCLV was introduced in \cite{ES} as a more direct
notion suitable for volume comparison theorems in the Lorentzian setting,
and we followed this line.
\end{remark}

The definition of
\emph{standard for comparisons of Lorentzian volumes
(SCLV)} was originally introduced in \cite{ES}.
(See Remark \ref{rm:vc} for more explanations.)

\begin{definition}[SCLV]
Let $(M,L)$ be a Finsler spacetime.
Let $x\in M$ and $U_x\subset M$ be a neighborhood of $x$. $U_x$ is said to be a SCLV at $x$ if there is a set $\wz U_x\subset T_xM$ satisfying that
\begin{enumerate}
\item $\wz U_x$ is an open set in the causal future of $x$,
(see Subsection \ref{ssc:causality} for the definition of
causal future);\label{sclv1}
\item $\wz U_x$ is \emph{star-shaped} from the origin, i.e. if $v\in\wz U_x$,
then $tv\in\wz U_x$, for all $t\in(0,1)$;\label{sclv2}
\item $\wz U_x$ is contained in a compact set in $T_xM$;\label{sclv3}
\item the exponential map at $x$ is defined on $\wz U_x$, and the restriction
of $\exp_x$ to $\wz U_x$ is a diffeomorphism onto its image
$U_x=\exp_x(\wz U_x)$.
\label{sclv4}
\end{enumerate}
\end{definition}
%


\subsection{Geodesics and curvatures}
Let $(M,L)$ be a Finsler spacetime. Let $\eta:[a,b] \lra M$
be a future-directed $C^1$-causal curve. Define the \emph{action}
\begin{align*}
\mathcal{S}(\eta):=\int_a^b L \big( \dot{\eta}(t) \big) \,\td t.
\end{align*}
The \emph{Euler--Lagrange equation} for $\mathcal{S}$ provides
the \emph{geodesic equation} (with the help of homogeneous
function theorem)
\begin{align}\label{eq:geod}
\ddot{\eta}^\alpha +\sum_{\beta,\gamma=0}^n
 \widetilde{\Gamma}^\alpha_{\beta \gamma}(\dot{\eta})
 \dot{\eta}^\beta \dot{\eta}^\gamma =0,
\end{align}
where we define
\begin{align}\label{eq:gamma}
\widetilde{\Gamma}^\alpha_{\beta \gamma}(v)
:=\frac{1}{2} \sum_{\delta=0}^n g^{\alpha \delta}(v)
 \bigg( \frac{\del g_{\delta \gamma}}{\del x^\beta}
 +\frac{\del g_{\beta \delta}}{\del x^\gamma}
 -\frac{\del g_{\beta \gamma}}{\del x^\delta} \bigg)(v)
\end{align}
for $v \in TM \setminus \{0\}$ and $(g^{\alpha \beta}(v))$
denotes the inverse matrix of $(g_{\alpha \beta}(v))$.
The equation \eqref{eq:geod} implies that $L(\dot{\eta})$
is constant.

\begin{definition}[Causal geodesics]\label{df:geod}
Let $\eta:[a,b]\lra M$ be a $C^{\infty}$-causal curve.
It is said to be a \emph{(future-directed) causal geodesic}
if \eqref{eq:geod} holds for all $t \in (a,b)$.
\end{definition}

\begin{remark}
Since $L(\dot{\eta})$ is constant, a causal geodesic is indeed
either a timelike geodesic or a lightlike geodesic.
By the basic ordinary differential equation theory,
given arbitrary $v \in \overline{\Omega}_x$,
there exists some $\ve>0$ and a unique $C^{\infty}$-geodesic
$\eta:(-\ve,\ve) \lra M$ satisfying $\dot{\eta}(0)=v$.
\end{remark}

\begin{definition}[Radial vector]
Let $(M,L)$ be a Finsler spacetime and $x\in M$.
A tangent vector $v\in T_xM$ is said to be \emph{radial} if there
is a geodesic $\eta:[0,T]\lra M$ with $\eta(T)=x$ such that
$v=\dot\eta(T)$. Let $U_x$ be a SCLV at $x$.
A tangent plane $\pi$ to $U_x$ is said to be \emph{radially timelike}
if it contains a timelike radial vector.
\end{definition}

\begin{definition}[Exponential map]\label{eq:exp}
Given $v \in \overline{\Omega}_x$, if there is a geodesic $\eta:[0,1] \lra M$ with $\dot{\eta}(0)=v$,
then we define $\exp_x(v):=\eta(1)$.
\end{definition}

We now define the \emph{geodesic spray coefficients}
and the \emph{nonlinear connection} as
\begin{equation}\label{eq:G&N}
G^\alpha(v) :=\sum_{\beta,\gamma=0}^n
 \widetilde{\Gamma}^\alpha_{\beta \gamma}(v) v^\beta v^\gamma, \qquad
N^\alpha_\beta(v) := \frac{1}{2} \frac{\del G^\alpha}{\del v^\beta}(v)
\end{equation}
for $v \in TM \setminus \{0\}$, and $G^\alpha(0)=N^\alpha_\beta(0):=0$.
Note that $G^\alpha$ is positively $2$-homogeneous
and $N^\alpha_\beta$ is positively $1$-homogeneous,
and $2 G^\alpha(v) =\sum_{\beta=0}^n N^\alpha_\beta(v)v^\beta$ holds
by the homogeneous function theorem.
The geodesic equation \eqref{eq:geod} is now written as
$\ddot{\eta}^\alpha +2 G^\alpha(\dot{\eta})=0$.
In order to define the covariant derivative,
we need to modify $\widetilde{\Gamma}^\alpha_{\beta \gamma}$ in \eqref{eq:gamma} as
\begin{equation}\label{eq:Gamma}
\Gamma^\alpha_{\beta \gamma}(v)
 :=\widetilde{\Gamma}^\alpha_{\beta \gamma}(v) -\frac{1}{2} \sum_{\delta,\mu=0}^n
 g^{\alpha \delta}(v) \bigg( \frac{\del g_{\delta \gamma}}{\del v^\mu} N^\mu_\beta
  +\frac{\del g_{\beta \delta}}{\del v^\mu} N^\mu_\gamma
 -\frac{\del g_{\beta \gamma}}{\del v^\mu} N^\mu_\delta \bigg)(v)
\end{equation}
for $v \in TM \setminus \{0\}$. Note that $\Gamma_{\beta\gamma}^\alpha$
could be used for defining the \emph{Chern connection} in Finsler
spacetimes.

\begin{definition}[Covariant derivative]\label{df:cov}
Let $X$ be a $C^1$-vector field on $M$,
$x \in M$ and $v,w \in T_xM$ with $w \neq 0$.
Define the \emph{covariant derivative} of $X$ by $v$ with
reference (support) vector $w$ by
\begin{equation}\label{eq:covd}
D^w_v X :=\sum_{\alpha ,\beta=0}^n \bigg\{ v^\beta \frac{\del X^\alpha}{\del x^\beta}(x)
 +\sum_{\gamma=0}^n \Gamma^\alpha_{\beta \gamma}(w) v^\beta X^\gamma(x) \bigg\}
 \frac{\del}{\del x^\alpha} \Big|_x.
\end{equation}
\end{definition}


\begin{definition}[Parallel vector field]
A vector field $V$ along a curve $\eta:I\longrightarrow M$ is said to be
\emph{$g_{\dot\eta}$-parallel} if
\begin{align*}
D_{\dot\eta}^{\dot\eta}V(t)=0,
\end{align*}
for all $t\in I$.
\end{definition}

As usual, one can define Jacobi fields
along a causal geodesic by considering a variation.
We here just employ the needful notations for completeness.

Define
\begin{align*}
R^\alpha_\beta(v) :=\frac{\del G^\alpha}{\del x^\beta}(v)
 -\sum_{\gamma=0}^n \lf[ \frac{\del N^\alpha_\beta}{\del x^\eta}(v) v^\gamma
 -\frac{\del N^\alpha_\beta}{\del v^\gamma}(v) G^\gamma(v) \r]
 -\sum_{\gamma=0}^n N^\alpha_\gamma(v) N^\gamma_\beta(v),
\end{align*}
for $v \in \overline{\Omega}$ ($R^\alpha_\beta(0)=0$).
Then, the \emph{Riemannian curvature} is defined as
\begin{equation}\label{eq:R_v}
R_v(w) :=\sum_{\alpha,\beta=0}^n R^\alpha_\beta(v) w^\beta \frac{\del}{\del x^\alpha}\Big|_x,
\end{equation}
for $v \in \overline{\Omega}_x$ and $w \in T_xM$.

\begin{definition}[Jacobi fields]\label{df:Jacobi}
Let $\eta:[a,b]\longrightarrow M$ be a causal geodesic.
A smooth vector field $Y:[a,b]\longrightarrow TM$ along
$\eta$ is said to be a \emph{Jacobi field} if $Y$ satisfies
the \emph{Jacobi equation},
\begin{align}\label{eq:Jacobi}
D^{\dot{\eta}}_{\dot{\eta}} D^{\dot{\eta}}_{\dot{\eta}} Y +R_{\dot{\eta}}(Y) =0.
\end{align}
\end{definition}

\begin{definition}[Conjugate points]\label{df:conj}
Let $\eta:[a,b] \lra M$ be a nonconstant causal geodesic.
If there is a nontrivial Jacobi field $Y$ along $\eta$ such that $Y(a) =Y(t) =0$ for some $t \in (a,b]$,
then we call $\eta(t)$ a \emph{conjugate point} of $\eta(a)$ along $\eta$.
\end{definition}

Equivalently, $\eta(t)$ is conjugate to $\eta(a)$ if
$\td(\exp_{\eta(a)})_{(t-a)\dot{\eta}(a)}:
T_{(t-a)\dot{\eta}(a)}(T_{\eta(a)}M) \lra T_{\eta(t)}M$
does not have full rank.

The flag and Ricci curvatures are defined by using $R_v$ in \eqref{eq:R_v} as follows.
The flag curvature corresponds to the sectional curvature in the Riemannian context.

\begin{definition}[Flag curvature]\label{df:curv}
For $v \in {\Omega_x} $ and $w \in T_xM$ linearly independent of $v$,
define the \emph{flag curvature} of the plane $v \wedge w$ spanned by $v,w$
with \emph{flagpole} $v$ as
\begin{equation}\label{eq:flag}
\bK(v,w) :=-\frac{g_v(R_v(w),w)}{g_v(v,v) g_v(w,w) -g_v(v,w)^2}.
\end{equation}
\end{definition}

We remark that this is the opposite sign to \cite{BEE},
while the Ricci curvature will be the same. Note that,
for $v$ timelike, the denominator in the right-hand side
of \eqref{eq:flag} is negative. The flag curvature is
not defined for $v$ lightlike, for in this case the
denominator could vanish. We define the Ricci curvature
directly as the trace of $R_v$ in \eqref{eq:R_v}.

\begin{definition}[Ricci curvature]\label{df:Ric}
For $v \in \overline{\Omega}_x \setminus \{0\}$,
the \emph{Ricci curvature} or \emph{Ricci scalar} is defined by
\begin{align*}
\Ric(v) :=\trace(R_{v}).
\end{align*}
\end{definition}

Since $\Ric(v)$ is positively $2$-homogeneous,
we can set $\Ric(0):=0$ by continuity.
We say that $\Ric \ge K$ holds \emph{in timelike directions} for some $K \in \R$
if we have $\Ric(v) \ge KF(v)^2 =-KL(v)$ for all $v \in \Omega$,
where $F(v):=\sqrt{-g_v(v,v)}=\sqrt{-L(v)}$.

For a normalized timelike vector $v \in \Omega_x$ with $F(v)=1$,
$\Ric(v)$ can be given as
\[ \Ric(v) =\sum_{i=1}^{n} \bK(v,e_i), \]
where $ \{v\} \cup \{e_i\}_{i=1}^{n}$ is an orthonormal basis with respect to $g_v$,
namely $g_v(e_i,e_j)=\delta_{ij}$ and $g_v(v,e_i)=0$ for all $i,j=1,\ldots,n$.

\begin{definition}[Weighted Ricci curvature]\label{def:RicN} 
Let $(M,L,\rho)$ be a Finsler spacetime with $\dim M=1+n$,
where $\rho$ is an arbitrary positive $C^\infty$ measure on $M$.
Given a smooth causal vector field $V$ such that all integral
curves are geodesic, we can always decompose $\rho$ in local coordinates as
\begin{align*}
\td\rho=e^{-\psi(V(x))}\sqrt{-\det\lf[\lf(g_{\alpha\beta}(V(x))
\r)_{\alpha,\beta=0}^n\r]}\td x_0\cdots\td x_n,
\end{align*}
where $\psi:TM\setminus\{0\}\lra\mathbb{R}$ is a positively 0-homogeneous smooth function called the \emph{weight function corresponding to the measure $\rho$}. 
For a causal geodesic $\eta(t)=\exp_x(tv)$ with $t\in(-\vez,\vez)$
and $v\in\overline{\Omega}_x\setminus \{0\}$, denote the
\emph{weight function along $\eta$} as $\psi_\eta(t):=\psi(\dot\eta(t))$.
Then, for $N \in \R \setminus \{n\}$,
we define the \emph{weighted Ricci curvature} by
\begin{equation}\label{eq:Ric_N}
\Ric_N(v) :=\Ric(v) +\psi_\eta''(0) -\frac{\psi_\eta'(0)^2}{N-n}.
\end{equation}
For the cases of $N \to +\infty$ and $N=n$, we also define
\begin{align*}
\Ric_\infty(v) &:=\Ric(v)+  \psi_\eta''(0) ,\\
\Ric_n(v) &:=\begin{cases}
\Ric(v)+    \psi_\eta''(0)  & \textrm{ if }\psi_\eta'(0)=0,\\
-\infty & \textrm{ if }\psi_\eta'(0) \neq 0. 
\end{cases}
\end{align*}
\end{definition}

\begin{remark}
The weighted Ricci curvature $\Ric_N$ might also be called
the \emph{Bakry--\'Emery--Ricci curvature}, due to the
pioneering work by Bakry and \'Emery \cite{BE} in the Riemannian
situation (we refer to the book \cite{BGL} for further information).
The Finsler version was introduced in \cite{Oint} as we mentioned,
and we refer to \cite{Ca} for the case of Lorentzian manifolds.
The weighted version is essential in our Lorentz--Finsler setting,
due to the possible lack of a canonical measure (like the volume
measure in the Lorentzian case), see \cite{ORand} for the
positive-definite case.
\end{remark}%

\subsection{Jacobi tensor field}
\begin{definition}[Jacobi tensor field]\label{df:jtf}
Let $(M,L)$ be a Finsler spacetime.
Let $\eta:I\longrightarrow M$ be a timelike geodesic of unit speed.
Define
\begin{align*}
N_\eta(t):=\{v\in T_{\eta(t)}M: g_{\dot\eta(t)}(v,\dot\eta(t))=0\}.
\end{align*}
An endomorphism $\sJ(t):N_\eta(t)\longrightarrow N_\eta(t)$
for each $t\in I$ is called a \emph{Jacobi tensor field}
along $\eta$ if
\begin{enumerate}
\item for any $g_{\dot\eta}$-parallel vector field $P$ along $\eta$,
$Y_P^\sJ(t):=\sJ(t)(P(t))$ is a Jacobi field; and
\item for any $t\in I$, $\ker(\sJ(t))\cap\ker(\sJ'(t))=\{0\},$ where
$\sJ'(t)$ is defined by $\sJ'(t)(P(t)):=D_{\dot\eta}^{\dot\eta}
Y_P^{\sJ}(t)$ for any $g_{\dot\eta}$-parallel
vector field $P$ along $\eta$.
\end{enumerate}
Furthermore, a Jacobi tensor field $\sJ$ is said to be
\emph{$g_{\dot\eta}$-symmetric}
if, for any $g_{\dot\eta}$-parallel vector fields $P_1$ and $P_2$,
it holds that $g_{\dot\eta}(P_1,\sJ(P_2))=g_{\dot\eta}(\sJ(P_1),P_2).$
\end{definition}

\begin{remark}
It is easy to show that $\sJ(t)(w)\in N_\eta(t)$ if $w\in N_\eta(t)$. Therefore,
a Jacobi tensor field is well defined. Besides, the derivative is also well defined
in the sense that $\sJ'(t):N_\eta(t)\lra N_\eta(t)$, since, for any
$g_{\dot\eta}$-parallel vector field $P$,
\begin{align*}
g_{\dot\eta}(\sJ'(t)(P(t)),\dot\eta)
=\frac{\td}{\td t}\lf(g_{\dot\eta}(\sJ(t)(P(t)),\dot\eta)\r)
-g_{\dot\eta}(\sJ(t)(P(t)),D_{\dot\eta}^{\dot\eta}\dot\eta(t))=0.
\end{align*}
\end{remark}
\begin{remark}
With the above definitions, we can deduce that, for any $g_{\dot\eta}$-parallel
vector field $P$ along $\eta:I\longrightarrow M$, and any $t\in I$,
\begin{align}\label{eqn:jacobi}
\sJ''(t)(P(t))+R_{\dot\eta(t)}(\sJ(t)(P(t)))
=(Y^{\sJ}_P)''(t)+R_{\dot\eta(t)}(Y^\sJ_P(t))=0,
\end{align}
which is simplified in some references as $\sJ''+\sR\sJ=0$. See \cite{BEE}.
\end{remark}

The following results are fundamental and frequently
used in our proofs. We here just give some simple explanations.
\begin{lemma}\label{prop:dexp}
Let $\sA$ be a Jacobi tensor field along $\eta:[0,T]\longrightarrow M$
with $\sA(0)=\sO_n$ and $\sA'(0)=\sI_n$.
Then we have $\sA(t)(P_i(t))=(\td\exp_x)_{tv}(te_i)$, where
$\{e_i\}_{i=1}^n\cup\{v\}$ forms an orthonormal basis of $(T_xM,g_v)$
with $\dot\eta(0)=v$, and $P_i(t)$ is obtained by extending $e_i$
parallelly along with $P_i(0)=e_i$,
for all $i=1,\ldots,n$. Furthermore,
\begin{align*}
\det\left(\td\exp_x\right)_{tv}=t^{-n}\det(\sA(t)),
\end{align*}
under the basis $\{P_i(t)\}_{i=1}^n$ of $N_{\eta}(t)$.
\end{lemma}

\begin{proof}
Let $J_i(t):=(\td\exp_x)_{tv}(te_i)$ and $Y_i(t):=\sA(t)(P_i(t))$,
for any $t\in[0,T]$ and $i=1,\ldots,n$. Since $Y_i(0)=J_i(0)=0$ and
$Y_i'(0)=J_i'(0)=e_i$, we know that $Y_i=J_i$.
Then, for any $i=1,\ldots,n$,
\begin{align*}
\sA(t)\left(P_i(t)\right)=(\td\exp_x)_{tv}(te_i)
=t\left(\td\exp_x\right)_{tv}(e_i),
\end{align*}
which shows that the representation matrix of $\sA(t)$ is the same as
$t(\td\exp_x)_{tv}$. Therefore,
\begin{align*}
\det \sA(t)=t^n\det\left(\td\exp_x\right)_{tv},
\end{align*}
which completes the proof.
$\qedd$
\end{proof}

%

\section{Bishop--Gromov's theorem}\label{sc:BG}
In this section, for $x\in M$ we denote by $\Fut_1(x)$ the \emph{set of all unit,
i.e. $F(v):=\sqrt{-g_v(v,v)}=\sqrt{-L(v)}=1$, future-directed
timelike vectors $v\in T_xM$ such that $\exp_x(v)$ is defined} and by $\Gamma_1(x)\supset\Fut_1(x)$
the \emph{set of all unit timelike geodesics starting
from $x\in M$.}
The \emph{cut function of $U_x$} is defined by
$\bt_{U_x}(v):=\sup\{t\in(0,\infty):tv\in
\wz U_x\}$, for all $v\in\Fut_1(x).$ If $tv\notin \wz U_x$ for
any $t\in(0,\infty)$, then let $\bt_{U_x}(v)=\infty$.
Define $\wz U_1:=\{v\in\Fut_1(x):\bt_{U_x}(v)<\infty\}$
and $\bt_x:=\inf_{v\in\Fut_1(x)}\bt_{U_x}(v)$ if there
occurs no confusion. For all $r\in(0,1]$,
we also define $\wz U_x(r):=\{rv:v\in \wz U_x\}$,
$U_x(r):=\exp_x(\wz U_x(r))$.

The following results are useful in our proof.
\begin{lemma}\label{lm:gromov}
Suppose $f$ and $g$ are positive integrable functions, of a real variable $r$, for which $f/g$ is non-increasing with respect to $r$. Then the function 
\begin{align*}
\frac{\int_0^rf(t)\,\td t}{\int_0^rg(t)\,\td t}
\end{align*}
is also non-increasing with respect to $r$.
\end{lemma}
\begin{proof}
The proof is the same as that in \cite[Lemma III.4.1]{Ch}).
$\qedd$
\end{proof}

The following proposition is standard and useful in
the proof of Theorem \ref{thm:vc}. Recall the definition
of a $g_{\dot\eta}$-symmetric Jacobi tensor in
Definition \ref{df:jtf}.
`

\begin{proposition}\label{prop:hric}
Let $\sA$ be a symmetric Jacobi tensor field along
a geodesic $\eta$ with $\dot\eta(0)\in\Gamma_1(x)$ and
satisfy that $\det(\sA(t))>0$ for $t\in(0,T)$, where $T>0$. Let $N\in
(-\infty,0)\cup(n,\infty)$ and
\begin{align*}
h(t):=e^{-\psi_\eta(t)/N}\left(\det(\sA(t))\right)^{1/N}.
\end{align*}
If, for some $t\in(0,T)$, $\Ric_N(\dot{\eta}(t))\ge c$, where
$c\in\R$, then
\begin{align*}
Nh''(t)+ch(t)\le0.
\end{align*}
\end{proposition}

\begin{proof}
Let
\begin{align*}
h_1(t):=e^{-\frac{\psi_\eta(t)}{N-n}},\qquad h_2(t):=\left(\det\sA(t)\right)^{1/n}.
\end{align*}
Then,
\begin{align*}
h=h_1^{\frac{N-n}N}h_2^{\frac nN},
\end{align*}
which implies that
\begin{align}\label{eq:ddh1}
Nh''h^{-1}=(N-n)h_1^{-1}h_1''+nh_2^{-1}h_2''-\frac{(N-n)n}N
\left(h_1^{-1}h_1'-h_2^{-1}h_2'\right)^2.
\end{align}
It can be easily calculated that, for any $t\in[0,T],$
\begin{align}\label{eq:ddh2}
\frac{h_1''(t)}{h_1(t)}=\left(\frac{\psi'_\eta(t)}{N-n}\right)^2
-\frac{\psi_\eta''(t)}{N-n}.
\end{align}
Combining \eqref{eq:ddh1}, \eqref{eq:ddh2} with
$N\in(-\infty,0)\cup(n,\infty)$, we have
\begin{align*}
Nh''(t)h^{-1}(t)\le\frac{\left(\psi'_\eta(t)\right)^2}{N-n}-\psi_\eta''(t)-\Ric(\dot\eta(t))
=-\Ric_N(\dot\eta(t))\le-c,
\end{align*}
where we used the curvature bound in the last inequality
and the unweighted Bishop inequality, i.e.
\begin{align*}
(\tr(A'A^{-1}))'\le-\Ric(\dot \eta)-\frac{(\tr(A'A^{-1}))^2}n,
\end{align*}
the proof of which is the same as that of \eqref{eq:Riccati_trC}.
$\qedd$
\end{proof}

\emph{Proof of Theorem \ref{thm:vc}. }
Fix $v\in\wz U_1$, then $\bt_{U_x}(v)=\bt_x$ by our assumption.
Put $\eta(t):=\exp_x(tv)$ for all $t\in[0,\bt_x)$.
Choose a $g_v$-orthonormal basis $\{e_i\}_{i=1}^n\cup\{v\}$ of $T_xM$.
Let $\{E_i(t)\}_{i=0}^n$ be a frame
along $\eta$ such that $E_i$ is $g_{\dot\eta}$-parallel with $E_0(0)=v$ and $E_i(0)=e_i$ for $i=1,\ldots,n$.
Then, $\{E_i(t)\}_{i=0}^n$ forms an orthonormal basis of $T_{\eta(t)}M$, and
hence, for all $t\in[0,\bt_x)$,
\begin{align}\label{eq:detg}
\det[g_{\dot\eta(t)}(E_i(t),E_j(t))]=-1.
\end{align}
Moreover, $\{E_i(t)\}_{i=1}^n$ is a basis of $N_\eta(t)$.
For any $w\in N_\eta(t)$,
we extend it to the $g_{\dot\eta}$-parallel vector field $P$ along
$\eta$ such that $P(t)=w$ and define
\begin{align*}
\sA(t)(w):=\left(\td(\exp_x)\right)_{tv}\left(tP(0)\right).
\end{align*}
Then, for any $g_{\dot\eta}$-parallel vector field $Q$ along $\eta$, we have
\begin{align*}
Y_Q^\sA(t):=\sA(t)(Q(t))
\end{align*}
is a Jacobi field and if there is some $t_0\in [0,\bt_x)$ such that
$Y^\sA_Q(t_0)=0$ and $(Y^\sA_Q)'(t_0)=0$,
then $Q(t)\equiv0$ for all $t\in [0,\bt_x)$. This
shows that $\sA$ is a Jacobi tensor field.

On the other hand, it is easy to check that $\sA(0)=\sO_n$, $\sA'(0)=\sI_n$,
and under the basis $\{E_i(t)\}_{i=1}^n$ of $N_\eta(t)$, $\det(\sA(t))>0$
for $t\in (0,\bt_x)$. Let $\psi_\eta$ be the weight function of
$\rho$ along $\eta$ and $h_v(t):=e^{-\psi_{\eta}(t)/N}(\det(\sA(t)))^{1/N}$.
We use Proposition \ref{prop:hric} for $\Ric_N\ge c$, to obtain that
\begin{align*}
h''_v(t)+\frac cNh_v(t)\le0,
\end{align*}
for all $t\in(0,\bt_x)$,
which together with $\bs_{c/N}''(t)+(c/N)\bs_{c/N}(t)=0$,
and $\bs_{c/N}(t)\ge0$ for $t\in[0,T_x]$,
we can infer that
\begin{align*}
\frac {\td}{\td t}\left[h'_v(t)\bs_{c/N}(t)-h_v(t)\bs_{c/N}'(t)\right]\le0.
\end{align*}
Since $h_v(0)=0$ and $N\in(n,\infty)$, we know that
\begin{align*}
\lim_{t\to0^+}h'_v(t)\bs_{c/N}(t)
&=\lim_{t\to0^+}\frac{h_v(t)\bs_{c/N}(t)}t
=\lim_{t\to0^+}e^{-\psi_\eta(t)/N}
\lim_{t\to0^+}\left(\frac{\det(\sA(t))}{t^N}\right)^{1/N}
\frac{\bs_{c/N}(t)}{t}t\\
&=\lim_{t\to0^+}e^{-\psi_\eta(t)/N}
\lim_{t\to0^+}\left(\frac{\det(\sA(t)/t)}{t^{N-n}}\right)^{1/N}
\frac{\bs_{c/N}(t)}{t}t\\
&=\lim_{t\to0^+}e^{-\psi_\eta(t)/N}
\lim_{t\to0^+}\left(\det(\sA'(0))\right)^{1/N}t^{n/N}=0,
\end{align*}
which implies that the function $h_v/\bs_{c/N}$ is non-increasing in $(0,\bt_x)$.
By Lemma \ref{lm:gromov}, we have
\begin{align*}
\frac{\rho\left(U_x(r)\right)}{\rho\left(U_x(R)\right)}
=\frac{\int_{\wz U_1}\int_0^{r\bt_x}h_v^N(t)\,\td t\td\sigma(v)}{\int_{\wz
U_1}\int_0^{R\bt_x}h_v^N(t)\,\td t\td\sigma(v)}
\ge\frac{\int_0^{r\bt_x}\bs_{c/N}^N(t)\,\td t}{\int_0^{R\bt_x}\bs_{c/N}^N(t)\,\td t},
\end{align*}
where we employ Lemma \ref{prop:dexp} and \eqref{eq:detg},
together with that $\bt_{U_x}$ is constant on $U_1$, to
deduce that
\begin{align*}
\rho(U_x(r))&=\int_{(\exp_x)^{-1}(U_x(r))}e^{-\psi(u)}
\det\left(\td(\exp_x)\right)_u\,\td u\\
&=\int_{\wz{U}_1}\int_0^{r\bt_x}e^{-\psi_{\eta_v}(t)}\det(\sA(t))\,\td t\td\sigma(v)
=\int_{\wz{U}_1}\int_0^{r\bt_x}\left(h_v(t)\right)^N\,\td t\td\sigma(v),
\end{align*}
which completes the proof.
$\qedd$

\section{G\"unther's Theorem}\label{sc:Gn}
\emph{Proof of Theorem \ref{thm:vc-flag}.}
For any future-directed timelike unit vector $v\in\wz U_1$,
and a geodesic $\eta_v$ starting at $x$ with
$\dot\eta_v(0)=v$, note that $\eta_v$ has no conjugate
points in $[0,\bt_{U_x}(v))$, we can define, for all $t\in[0,\bt_{U_x}(v)),$
\begin{align*}
\sA_v(t):=N_{\eta_v}(t)&\longrightarrow N_{\eta_v}(t)\\
w&\mapsto \left(\td\exp_x\right)_{tv}\left(tP_w(0)\right),
\end{align*}
where $P_w$ is the $g_{\dot\eta_v}$-parallel vector field along
$\eta_v$ with $P_w(t)=w$. Then $\sA_v$ is a Jacobi tensor field
for the same reason as in the proof of Theorem \ref{thm:vc}. Now let
$\{e_i\}_{i=1}^n\cup\{v\}$ be orthonormal and
\begin{align*}
f(t):=\frac{\det\sA_v(t)}{\bs_{-c}^n(t)},
\end{align*}
for all $t\in(0,\bt_{U_x}(v))$. We claim that $f\ge1$ in $(0,\bt_{U_x}(v))$.
Indeed, by noting that
\begin{align*}
\lim_{t\to0^+}f(t)=\lim_{t\to0^+}\frac{\det\sA_v(t)}{t^n}\frac{t^n}{\bs_{-c}^n(t)}
=\det\sA'_v(0)=1,
\end{align*}
we need only to prove $f'\ge0$, which is equivalent to
\begin{align*}
\frac{\left(\det\sA_v(t)\right)'}{\det\sA_v(t)}\ge
n\frac{\bs_{-c}'(t)}{\bs_{-c}(t)}.
\end{align*}
Let $\sB_v:=\sA_v^T\sA_v.$ Then
\begin{align*}
\tr(\sB_v'\sB_v^{-1})=\frac{\left(\det\sB_v(t)\right)'}{\det\sB_v(t)}
=2\frac{\left(\det\sA_v(t)\right)'}{\det\sA_v(t)}.
\end{align*}
Since $\sB_v$ is symmetric, by Gram--Schmidt process,
we can find another orthonormal basis, without loss of
generality, still denoted by $\{e_i(t)\}_{i=1}^n$ of $N_{\eta_v}(t)$ for any
$t\in(0,\bt_{U_x}(v))$ such that $\sB_v(t)e_i(t)=\lambda_i(t)e_i(t)$.
Observe that
\begin{align*}
\lambda_i=g_{\dot\eta_v}\left(\sB_v e_i,e_i\right)
=g_{\dot\eta_v}\left(\sA_v e_i,\sA_v e_i\right)
=g_{\dot\eta_v}\left(Y_{e_i}^{\sA_v},Y_{e_i}^{\sA_v}\right),
\end{align*}
with the notations in Definition \ref{df:jtf}, we can obtain that
\begin{align*}
\tr\left(\sB_v'\sB_v^{-1}\right)
=\sum_{i=1}^n\lambda_i'\lambda_i^{-1}
=\sum_{i=1}^n
\frac{\left(g_{\dot\eta_v}\left(Y_{e_i}^{\sA_v},Y_{e_i}^{\sA_v}\right)
\right)'}{g_{\dot\eta_v}\left(Y_{e_i}^{\sA_v},Y_{e_i}^{\sA_v}\right)}.
\end{align*}
To simplify our notations, we denote that $Y_i:=Y_{e_i}^{\sA_v}$
and $\eta:=\eta_v$.
Since $Y_i(t)\neq0$ for all $t\in(0,\bt_{U_x}(v))$, we have
\begin{align*}
\frac{\td^2}{\td t^2}\left[\sqrt{g_{\dot\eta}(Y_i,Y_i)}\right]
&=\frac{\td}{\td t}\left[\frac{g_{\dot\eta}(Y_i,Y_i')}
{\sqrt{g_{\dot\eta}(Y_i,Y_i)}}\right]
=\frac{g_{\dot\eta}(Y_i'',Y_i)+g_{\dot\eta}(Y_i',Y_i')}
{\sqrt{g_{\dot\eta}(Y_i,Y_i)}}-\frac{\left[g_{\dot\eta}(Y_i,Y_i')\right]^2}
{\left[g_{\dot\eta}(Y_i,Y_i)\right]^{3/2}}\\
&=-\frac{g_{\dot\eta}(\sR_{\dot\eta}(Y_i),Y_i)}{\sqrt{g_{\dot\eta}(Y_i,Y_i)}}
+\frac{g_{\dot\eta}(Y_i',Y_i')g_{\dot\eta}(Y_i,Y_i)
-\left[g_{\dot\eta}(Y_i,Y_i')\right]^2}{\left[g_{\dot\eta}(Y_i,Y_i)\right]^{3/2}}.
\end{align*}
By the Cauchy--Schwarz inequality, the second term is nonnegative, thus we have the following inequalities by using the curvature bound condition $\bK(\dot\eta,Y_i)\le-c$ for $c\ge0$:
\begin{align*}
\frac{\td^2}{\td t^2}\left[\sqrt{g_{\dot\eta}(Y_i,Y_i)}\right]
&\ge-\frac{g_{\dot\eta}(\sR_{\dot\eta}(Y_i),Y_i)}
{\sqrt{g_{\dot\eta}(Y_i,Y_i)}}
=\frac{\bK(\dot\eta,Y_i)\left[g_{\dot\eta}(\dot\eta,\dot\eta)
g_{\dot\eta}(Y_i,Y_i)-g_{\dot\eta}(\dot\eta,Y_i)^2\right]}
{\sqrt{g_{\dot\eta}(Y_i,Y_i)}}\\
&\ge \frac{c\left[g_{\dot\eta}(Y_i,Y_i)+g_{\dot\eta}
(\dot\eta,Y_i)^2\right]}
{\sqrt{g_{\dot\eta}(Y_i,Y_i)}}
\ge c\sqrt{g_{\dot\eta}(Y_i,Y_i)},
\end{align*}
which, together with the definition of $\bs_{-c}$, implies that
\begin{align*}
&\frac{\td}{\td t}\left[\frac{\td}
{\td t}\left(\sqrt{g_{\dot\eta}(Y_i(t),Y_i(t))}\right)
\bs_{-c}(t)-\sqrt{g_{\dot\eta}(Y_i(t),Y_i(t))}\bs_{-c}'(t)\right]\\
&\hs\hs=\frac{\td^2}{\td t^2}\left(\sqrt{g_{\dot\eta}(Y_i(t),Y_i(t))}\right)
\bs_{-c}(t)-c\sqrt{g_{\dot\eta}(Y_i(t),Y_i(t))}\bs_{-c}(t)\ge0,
\end{align*}
for all $t\in(0,\bt_{U_x}(v))$.
Combining this with $Y_i(0)=0$, we obtain that
\begin{align*}
\frac{\left(\det\sA_v(t)\right)'}{\det\sA_v(t)}
=\frac12\tr\left(\sB_v'(t)\sB^{-1}(t)\right)
=\frac12\sum_{i=1}^n\frac{\left[g_{\dot\eta}(Y_i,Y_i)\right]'(t)}
{g_{\dot\eta}(Y_i(t),Y_i(t))}
\ge\sum_{i=1}^n\frac{\bs_{-c}'(t)}{\bs_{-c}(t)}=\frac{n\bs_{-c}'(t)}{\bs_{-c}(t)},
\end{align*}
which completes the proof of the claim. Finally,
\begin{align*}
\rho(U_x)
&=\int_{\wz U_1}\int_0^{\bt_{U_x}(v)}
e^{-\psi_{\eta}(t)}\det\left(\sA_v(t)\right)\td t\td \sigma(v)\\
&\ge e^{-k}\int_{\wz U_1}\int_0^{\bt_x}\bs_{-c}^n(t)\td t\td \sigma(v)
=e^{-k}\sigma(\wz U_1)\int_0^{\bt_x}\bs_{-c}^n(t)\td t,
\end{align*}
where we used $\psi_\eta(t)\le k$ to complete the proof.
$\qedd$

\section{The case $N=\infty$}\label{sc:Ninf}
For $N=\infty$, following the lines in \cite{Sh,WW}, we
obtain the following volume comparison theorem.
\begin{theorem}\label{thm:vc-inf}
Let $(M,L,\rho)$ be a Finsler spacetime satisfying
$\Ric_\infty(v)\ge nc$ for some $c\in\R$ and for all unit timelike
radial vectors $v\in T_xM$ tangent to some SCLV subset $U_x\subset M$
at $x\in U$. Suppose that, for any causal geodesic
$\eta_v(t):=\exp_x(tv)$ with $v\in\wz U_1$,
it holds that $\psi_{\eta_v}'\ge-a$ along
$\eta_v$, for some $a\in\R$ and for all $t\in(0,\bt_{U_x}(v))$.
If $\bt_{U_x}$ is constant on $U_1$, then
\begin{align*}
\frac{\rho\left(U_x(r)\right)}{\rho\left(U_x(R)\right)}
\ge \frac{\int_0^{rT_x}e^{at}\bs_{c}^n(t)\,\td t}
{\int_0^{RT_x}e^{at}\bs_{c}^n(t)\,\td t},
\end{align*}
for all $0<r\le R\le 1$, where $T_x:=\bt_x$ if $c\le0$ and
$T_x:=\min\{\bt_x,\pi/(2\sqrt{c})\}$ if $c>0$.
\end{theorem}

\begin{proof}
Define $\{E_i(t)\}_{i=0}^n$, $N_\eta(t)$ and $\sA(t)$ as in the proof
of Theorem \ref{thm:vc}. Let $\sR(t):N_\eta(t)\lra N_\eta(t)$ be an
endomorphism such that $\sR(t)(w):=R_{\dot\eta(t)}(w)$, where
$R_{\dot\eta(t)}$ is the curvature operator as in \eqref{eq:R_v}.
Then, $\sR(t)$ is in fact a linear map, and hence a matrix under
the basis $\{E_i(t)\}_{i=1}^n$. Let $\sC(t):=\sA'(t)\sA^{-1}(t)$ for
$t\in(0,\bt_{U_x}(v))$. $\sC$ is well defined since $\sA$ is non-degenerate
for that $\eta$ has no conjugate points in $(0,\bt_{U_x}(v))$.
Since $\sA$ is a Jacobi tensor field, we have
\begin{align*}
\sA''+\sR\sA=0.
\end{align*}
Right multiply $\sA^{-1}$ to both sides in the above equation, we deduce that
$\sA''\sA^{-1}+\sR=0$, which gives
\begin{align*}
\sC'+\sC^2+\sR=\sA''\sA^{-1}+\sA'\lf(\sA^{-1}\r)'+\sA'\sA^{-1}\sA'\sA^{-1}
+\sR=\sA''\sA^{-1}+\sR=0.
\end{align*}
Taking the traces of both sides, with the help of Cauchy--Schwarz inequality,
we get a \emph{Riccati inequality} for $\tr(\sC)$ as
\begin{align}\label{eq:Riccati_trC}
\lf[\tr(\sC)\r]'+\frac{[\tr(\sC)]^2}n+\Ric(\dot\eta)\le0.
\end{align}
For $t\in(0,\bt_{U_x}(v))$, let
\begin{align*}
\lambda(t)&:=\lf[\log\lf(\det\sA(t)\r)\r]',\\
\lambda_\psi(t)&:=\lambda(t)-\psi'_\eta(t)=\lf(\log\lf[e^{-\psi_\eta(t)}
\det\sA(t)\r]\r)',\\
\lambda_c(t)&:=n\frac{\bs_{c}'(t)}{\bs_{c}(t)}.
\end{align*}
Recall that $\Ric_\infty(\dot\eta)=\Ric(\dot\eta)+\psi''_\eta\ge nc$.
By \eqref{eq:Riccati_trC}, we obtain that
\begin{align*}
\lambda'+\frac{\lambda^2}n+\Ric(\dot\eta)\le0,
\end{align*}
which, together with
\begin{align*}
\lambda_c'+\frac{\lambda_c^2}n+nc=0,
\end{align*}
implies that
\begin{align}
\lf[\bs_c^2(\lambda-\lambda_c)\r]'
&=2\bs_c\bs_c'(\lambda-\lambda_c)+\bs_c^2(\lambda'-\lambda_c')
=\bs_c^2\lf[\frac{2\bs_c'}{\bs_c}(\lambda-\lambda_c)+(\lambda'-\lambda_c')\r]\noz\\
&\le\bs_c^2\lf[\frac{2\lambda_c}{n}(\lambda-\lambda_c)
-\frac{\lambda^2-\lambda_c^2}n+\psi_\eta''\r]
=\bs_c^2\lf[-\frac{(\lambda-\lambda_c)^2}n+\psi_\eta''\r]\le\bs_c^2\psi_\eta''.
\label{eq:sc}
\end{align}
Integrating \eqref{eq:sc} from 0 to some $t\in(0,\bt_{U_x}(v))$ yields that
\begin{align*}
\bs_c^2(t)(\lambda(t)-\lambda_c(t))
\le\int_0^t\bs_c^2(\tau)\psi''_\eta(\tau)\td\tau
=\bs_c^2(t)\psi'_\eta(t)-\int_0^t(\bs_c^2)'(\tau)\psi_\eta'(\tau)\td\tau.
\end{align*}
Since $\psi_\eta'(\tau)\ge-a$ and $(\bs_c^2)'(\tau)>0$,
for $\tau\in(0,\pi/(2\sqrt{c}))$, we see that
\begin{align*}
\lambda_\psi(t)\le\lambda_c(t)+\frac{a}{\bs_c^2(t)}\int_0^t(\bs_c^2)'(\tau)\td\tau
=\lambda_c(t)+a,
\end{align*}
which is equivalent to
\begin{align*}
\lf[\log(e^{-\psi_\eta(t)}\det(\sA(t)))\r]'\le\lf[\log(e^{at}\bs_c^n(t))\r]'.
\end{align*}
By again Lemma \ref{lm:gromov}, we have
\begin{align*}
\frac{\rho(U_x(r))}{\rho(U_x(R))}
\ge \frac{\int_0^{rT_x}e^{at}\bs_c^n(t)\td t}{\int_0^{RT_x}e^{at}\bs_c^n(t)\td t},
\end{align*}
for $0<r\le R\le 1$. $\qedd$
\end{proof}

Using a similar method, we can deduce the following estimate
when $N=\infty$. Let us mention Sturm's original results
for metric measure spaces, see \cite[Theorem 4.26]{St}.
If $U$ is a SCLV at $x$, then
define $\fB_U^+(x,r):=\exp_x(\{v\in \wz U_x:F(v)<r\})$.
\begin{theorem}
Let $(M,L,\rho)$ be a Finsler spacetime and satisfy that $\Ric_\infty\ge c$.
If $U$ is a SCLV at $x\in M$, then
\begin{align*}
\rho(\fB_U^+(x,r))\le\rho(\fB_U^+(x,4\vez))+\sigma\lf(\wz U_1
\r)\int_{4\vez}^r
e^{C_0t-\frac{ct^2}2}\td t,
\end{align*}
sufficiently small $\vez>0$ and all $r>4\vez$, where
$C_0>0$ is a constant depending on $c$, $\vez$ and $x$. In particular,
when $c=0$, i.e. the timelike $\infty$-convergence condition (see
\cite[Definition 9.10]{LMO}) holds,
\begin{align*}
\rho(\fB_U^+(x,r))\le\rho(\fB_U^+(x,4\vez))
+\sigma\lf(\wz U_1\r)\frac{e^{C_0r}}{C_0}.
\end{align*}
\end{theorem}
\begin{proof}
Let $x\in M$ and $v\in \wz U_1$. Define $\eta$ and $\sA$
as in the above proof. Let $f(t):=e^{-\psi_\eta(t)}\det\sA(t)$
and $h(t):=(\det\sA(t))^{1/n}$. Then, by again the Bishop inequality,
see \cite[(5.4)]{LMO},
\begin{align*}
\lf(\log\lf[f(t)e^{ct^2/2}\r]\r)''
&=\lf[-\psi_\eta(t)+\frac{ct^2}2+\log(\det\sA(t))\r]''\\
&=\lf[-\psi_\eta'(t)+ct+\frac{nh'(t)}{h(t)}\r]'\\
&=-\psi_\eta''(t)+c+n\frac{h''(t)h(t)-[h'(t)]^2}{[h(t)]^2}\\
&\le-\psi_\eta''(t)+c-\Ric(\dot\eta(t))=-\Ric_\infty(\dot\eta(t))+c\le0.
\end{align*}
Therefore, for all $t>4\vez>0$,
\begin{align}\label{eq:logfe}
\log(f(2\vez)e^{2c\vez^2})\ge
\frac{t-2\vez}{t-\vez}\log\lf(f(\vez)e^{\frac{c\vez^2}2}\r)
+\frac{\vez}{t-\vez}\log\lf(f(t)e^{\frac{ct^2}2}\r).
\end{align}
Since $\lim_{t\to0}f(t)=0$, there exists $\delta$ such that,
for all $\vez\in(0,\delta)$, we have $\log(f(2\vez)e^{2c\vez^2})<0$.
For such $\vez>0$, we can deduce from \eqref{eq:logfe} that
\begin{align*}
\log\lf(f(t)e^{ct^2/2}\r)
\le-\frac{(t-2\vez)\log\lf(f(\vez)e^{c\vez^2/2}\r)}{t\vez}t
\le-\frac{\log\lf(f(\vez)e^{c\vez^2/2}\r)}{\vez}t
=:C_0t.
\end{align*}
Finally, from the above estimates, we get that
\begin{align*}
\rho(\fB_U^+(x,r))&\le\rho(\fB_U^+(x,4\vez))+\int_{\wz U_1}\int_{4\vez}^r
f(t)\td t\td\sigma(v)\\
&\le\rho(\fB_U^+(x,4\vez))+\sigma\lf(\wz U_1\r)\int_{4\vez}^r
e^{C_0t-\frac{ct^2}2}\td t,
\end{align*}
which is exactly what we need.
$\qedd$
\end{proof}%

{\bf Acknowledgement.} The author would like to express
his deep thanks to Professor Shin-ichi OHTA for a careful reading of
an earlier version of this paper and for making a number of useful
suggestions and comments, including Remark \ref{rm:vc}.

{\small

}

\end{document}